\documentclass[a4paper,11pt]{article}

\usepackage[top=3.0cm, bottom=3.0cm, inner=3.0cm, outer=3.0cm,
includefoot]{geometry}

\usepackage{verbatim}
\usepackage{caption}
\usepackage{url}
\usepackage{amsmath}
\usepackage{geometry}
\usepackage{amssymb}
\usepackage{amsmath}
\usepackage{graphicx}
\usepackage{amsthm}
\usepackage{bbm}
\usepackage{float}
\usepackage{color,soul}
\usepackage{hyperref}
\usepackage[T1]{fontenc}
\usepackage[utf8]{inputenc}
\usepackage{authblk}
\usepackage{booktabs}
\usepackage{longtable}
\usepackage{enumerate}
\usepackage{tikz}
\usetikzlibrary{arrows}
\usetikzlibrary{decorations.markings}
\usepackage{standalone}
\usepackage{indentfirst}
\newcommand{\eChar}{\begin{enumerate}[(i)]}
\newcommand{\eCharR}{\begin{enumerate}[(a)]}
\newcommand{\eBr}{\begin{enumerate}[(1)]}

\title
{
Criteria on forbidden subgraphs in the complements for positive Lin--Lu--Yau curvature
}

\author[1]{Kaizhe Chen\thanks{Email: ckz22000259@mail.ustc.edu.cn}}
\author[2]{Shiping Liu\thanks{Email: spliu@ustc.edu.cn}}
\author[3]{Zhe You\thanks{Email: y30231280@mail.ecust.edu.cn}}
\affil[1]{School of the Gifted Young, University of Science and Technology of China}
\affil[2]{School of Mathematical Sciences, University of Science and Technology of China}
\affil[3]{School of Mathematics, East China University of Science and Technology}
\date{}

\theoremstyle{plain}
\newtheorem{lemma}{Lemma}[section]
\newtheorem{theorem}[lemma]{Theorem}

\newtheorem{corollary}[lemma]{Corollary}

\theoremstyle{definition}

\newtheorem{claim}[lemma]{Claim}

\newtheorem{definition}[lemma]{Definition}

\numberwithin{equation}{section}

\begin{document}

\maketitle

\begin{abstract}
We investigate forbidden subgraph conditions in the complement of a graph that guarantee positive Lin--Lu--Yau curvature. In particular, we prove that every graph whose complement contains no $4$-cycles has positive Lin--Lu--Yau curvature, with the only exception of the $4$-vertex path. 
We further prove that, for any integer $t\ge2$, every graph on at least $\max\{t^2-2t+2, 8t\}$ vertices whose complement contains no $K_{2,t}$ has positive curvature. 
In addition, this lower bound on the number of vertices is optimal for $t\geq 10$. 
Finally, we construct examples showing that, in general, the forbidden subgraphs in these results cannot be replaced by cycles of length other than $4$ or by complete bipartite graphs $K_{s,t}$ with $s> 2$ and $t> 2$.
\end{abstract}



\section{Introduction}

The local-to-global principle is a fundamental method for analyzing the structure of spaces. In recent years, this perspective has gained significant attention in the setting of discrete spaces, especially graphs, through the development of discrete notions of Ricci curvature \cite{NR17}. Among these, a particularly influential notion is  originally proposed by Ollivier \cite{O09} and later modified by Lin--Lu--Yau \cite{LLY11}. Ollivier's coarse Ricci curvature, along with the closely related Lin--Lu--Yau curvature, is defined for pairs of vertices using the optimal transport distance between their respective neighborhoods. 

In this paper, we establish criteria for deciding whether a graph has positive Lin--Lu--Yau curvature. 
These criteria are formulated in terms of forbidding certain subgraphs from the complement graph. 
Previous criteria for positive Lin--Lu--Yau curvature, stated in terms of minimum vertex degree or vertex connectivity, were obtained by Hehl \cite{Hehl24, Hehl25} and by the authors \cite{CLY}. 
We note that a graph has positive Lin--Lu--Yau curvature if and only if it has positive Ollivier's coarse Ricci curvature with idleness parameter $p\in [1/2, 1)$  (see \cite{BCLMP18} or~\eqref{eq:bourne} below). 
Moreover, Lin--Lu--Yau curvature can be interpreted as a coupling condition arising in the analysis of convergence for the associated time-discrete Markov chains \cite{O09}, and is also related to the time-continuous counterpart of Ollivier's coarse curvature and corresponding convergence properties of the associated time-continuous Markov process \cite[Section 5]{MW19}.

Our work is partly motivated by the results of Cushing--Stone \cite[Theorem 1.3]{CS24} and Huang--He--Zhang \cite[Theorem 1]{HHZ25}. 
Their results can be reformulated as follows: A graph on $n\geq 3$ vertices has Lin--Lu--Yau curvature at least $1$ if and only if it can be obtained from the complete graph $K_n$ by removing a matching. 
Note that a matching is the simplest kind of forest. This naturally leads to the question: does deleting any forest, or more generally any $C_4$-free graph, from a complete graph still guarantee positive curvature?
The following theorem shows that the only counterexample is the 4-vertex path.

\begin{theorem}\label{main}
    Let $G$ be a graph such that the complement graph of $G$ has no cycles of length $4$. Then $G$ has positive Lin--Lu--Yau curvature unless $G$ is a path on $4$ vertices.
\end{theorem}


It is worth noting that a graph with no cycles of length $4$ is generally expected to have non-positive curvature at some pair of vertices. 
In fact, Lin and the third named author \cite{LinYou24} classified all positively curved graphs with no $4$-cycles and minimum vertex degree at least $2$, showing that there are only $6$ such graphs. Theorem \ref{main} may be viewed as a complementary result to theirs: Forbidding $4$-cycles in the complement leads to opposite curvature phenomenon.

We observe that Theorem \ref{main} does not extend to the case where one forbids cycles of length $k$ in the complement graph for $k=3$ or any $k>4$, see Section \ref{section:sharpness}.

On the other hand, a $4$-cycle is the complete bipartite graph $K_{2,2}$. 
If instead we forbid $K_{2,t}$ for $t\geq 2$ in the complement, we obtain the following result. 

\begin{theorem}\label{main2}
    Let $t\ge 2$ be an integer.
    Let $G$ be a graph on at least $\max\{t^2-2t +2,8t\}$ vertices such that the complement graph of $G$ contains no $K_{2,t}$. 
    Then $G$ has positive Lin--Lu--Yau curvature.
\end{theorem}

We emphasize that the condition on the graph size is essential. In fact, the lower bound $n\ge t^2-2t +2$ is sharp for $t\ge 10$, as shown in Section \ref{section:sharpness}.

A natural question is whether $K_{2,t}$ in Theorem \ref{main2} can be replaced by other complete bipartite graphs $K_{s,t}$ with $s>2$ and $t> 2$. However, the examples provided in Section \ref{section:sharpness} show that this is not possible.

Throughout the paper, we use the following notation. 
Let $G=(V,E)$ be a simple finite graph.
Denote by $\overline{G}$ the complement graph of $G$.
For any $x\in V$ and $A\subseteq V$, let $N_A(x)$ be the set of neighbors of $x$ in $A$ and let $d_x:=|N_V(x)|$ be its degree. 
A vertex $y\ne x$ which is not adjacent to $x$ is called a non-neighbor of $x$.
For any $S\subseteq V$, define $N_A(S):=\cup_{v\in S}N_A(v)$.
For any two vertices $x$ and $y$, we denote the distance between them by $\rho(x,y)$. 
Let $\delta_{xy}: V\times V\to [0,1]$ denote the function defined as follows:
    \[\delta_{xy}(a,b)=\left\{
                \begin{array}{ll}
                    1, & \hbox{if $a=x$, $b=y$;}\\
                    0, & \hbox{otherwise.}
                \end{array}
                \right.
    \]
For two fixed vertices $x$ and $y$, we set $N_{xy}:=N_V(x)\cap N_V(y)$, $A^{(xy)}_x:=N_V(x)\backslash (N_{xy}\cup \{ y \})$, $A^{(xy)}_y:=N_V(y)\backslash (N_{xy}\cup \{ x \})$, and $R_{xy}:=V\backslash (N_V(x)\cup N_V(y))$. 
For simplicity, we write $A_x$ for $A^{(xy)}_x$ and write $A_y$ for $A^{(xy)}_y$ hereafter.

\section{Preliminaries}
\subsection{Lin--Lu--Yau curvature}
Before we introduce the Lin--Lu--Yau curvature, we recall the definition of Wasserstein distance first.
\begin{definition}[Wasserstein distance]
     Let $G=(V,E)$ be a locally finite graph, $\mu_1$ and $\mu_2$ be two probability measures on $G$. The {\it Wasserstein distance} $W(\mu_1, \mu_2)$ between $\mu_1$ and $\mu_2$ is defined as
     \begin{align}\label{defi}
         W(\mu_1,\mu_2)=\inf_{\pi}\sum_{u\in V}\sum_{v\in V}\rho(u,v)\pi(u,v),
     \end{align}
     where the infimum is taken over all the mappings $\pi: V\times V\to [0,1]$ satisfying
     $$\mu_1(u)=\sum\limits_{v\in V}\pi(u,v) \text{ for any}\ u\in V$$
     and
     $$\mu_2(v)=\sum\limits_{u\in V}\pi(u,v) \text{ for any}\ v\in V.$$ 
     Such a mapping is called a {\it transport plan} from $\mu_1$ to $\mu_2$. 
     A transport plan that attains the infimum in \eqref{defi} is called {\it optimal}.
     We call a transport plan {\it simple} if, for each vertex $u\in V$, we have
     $$\pi(u,u)=\min\{ \mu_1(u), \mu_2(u)\}.$$
\end{definition}
 Here, for a given idleness parameter $p\in [0,1]$, we consider the particular measure around a vertex $x\in V$ as follows:
    \[\mu_x^p(y)=\left\{
                    \begin{array}{ll}
                      p, & \hbox{if $y=x$;} \\
                      \frac{1-p}{d_x}, & \hbox{if $xy\in E$;} \\
                      0, & \hbox{otherwise.}
                    \end{array}
                  \right.
     \]
     
   Based on the probability measure above, two kinds of Ricci curvature on graphs are defined as follows.
\begin{definition}[$p$-Ollivier curvature and Lin--Lu--Yau curvature] 
Let $G=(V,E)$ be a locally finite graph. 
For any vertices $x\neq y\in V(G)$, the {\it $p$-Ollivier curvature} $\kappa_p(x,y)$, $p\in [0,1]$, is defined as
     \[\kappa_p(x,y)=1-\frac{W(\mu_x^p,\mu_y^p)}{\rho(x,y)}.\]
The {\it Lin--Lu--Yau curvature} $\kappa_{LLY}(x,y)$ is defined as
     \[\kappa_{LLY}(x,y)=\lim_{p\to 1}\frac{\kappa_p(x,y)}{1-p}.\]
\end{definition}
The ratio $\kappa_p(x,y)/(1-p)$ is constant when $p$ is large enough. 
Indeed, it was proved in \cite{BCLMP18} that 
\begin{equation}\label{eq:bourne}
    \kappa_{LLY}(x,y)=\frac{\kappa_p(x,y)}{1-p} \,\,\text{for any $p \in \left[\frac{1}{\max\{d_x,d_y\}+1},1\right]$}.
\end{equation}

For any locally finite graph $G$, the  normalized graph Laplacian $\Delta$ is defined as
$$\Delta f(x):=\frac{1}{d_x} \sum_{y: xy\in E(G)}(f(y)-f(x)), \text{ for any $f: V(G)\to \mathbb{R}$ and any $x\in V(G)$}.$$
There is another limit-free formulation of Lin--Lu--Yau curvature, which was given by M\"{u}nch and Wojciechowski \cite{MW19}. 
\begin{theorem}[Curvature via the Laplacian {\cite[Corollary 2.2]{MW19}}]\label{Curvature via the Laplacian}
    Let $G$ be a locally finite graph and let $xy$ be an edge. Then
    $$\kappa_{LLY}(x, y)=\inf _{\substack{f:N_V(x)\cup N_V(y)\to \mathbb{Z}\\f \in Lip(1) \\ f(y)-f(x)=1}} \left(\Delta f(x)-\Delta f(y)\right).$$
\end{theorem}

The following lemma plays an important role in our proof.
\begin{lemma}[{\cite[Lemma 4.1]{CLY}}]\label{simple}
    Let $G$ be a graph. 
    Let $x$ and $y$ be two adjacent vertices in $G$ with $d_x\ge d_y$. 
    For any $p\in \left[ \frac{1}{1+d_y},1 \right]$, there is a simple optimal transport plan $\pi$ from $\mu_x^p$ to $\mu_y^p$ such that $$\pi(x,y)=p-\frac{1-p}{d_y}.$$ 
\end{lemma}

\subsection{Matching}
A {\it matching} in a graph is a set of pairwise non-adjacent edges. 
A matching $M$ is called {\it a matching of} $U\subseteq V$ if every vertex in $U$ is incident with an edge in $M$.
The following classical theorem (see, e.g., \cite[ Theorem 2.1.2]{Diestel}) is useful in our proof.
\begin{theorem}[Hall's Marriage Theorem] \label{Marriage}
    Let $G$ be a bipartite graph with partition $A\cup B$. 
    Then, $G$ contains a matching of $A$ if and only if $|N_B(S)|\geq |S| \,\,\text{for  all}\,\,S\subseteq A$.
\end{theorem}

\section{Basic results}
In this section, we present some preliminary lemmas.
First, we provide a lemma which will be used repeatedly in our proof of Theorem \ref{main} and Theorem \ref{main2}.
\begin{lemma}\label{adjust}
    Let $G$ be a graph and $xy$ be an edge in $G$. 
    Let $p\in [0,1]$, and let $\pi$ be an optimal transport plan from $\mu_x^p$ to $\mu_y^p$. 
    For any vertices $\{u,v,z,w\}\subseteq V$ with $\pi(u,v)>0$ and $\pi(z,w)>0$, we have 
    \begin{align*}
        \rho(u,v)+\rho (z,w)\le \rho(u,w)+\rho (z,v).
    \end{align*}
\end{lemma}

\begin{proof}
    Let $\epsilon:=\min\{\pi(u,v),\pi(z,w)\}$. 
    Then $\epsilon>0$. 
    Consider the transport plan $\pi_1$ defined as $$\pi_1:=\pi-\epsilon(\delta_{zw}+\delta_{uv}-\delta_{zv}-\delta_{uw}).$$
        It follows that
        \begin{align*}
            \sum_{a,b\in V}\rho(a,b)\pi(a,b)-\sum_{a,b\in V}\rho(a,b)\pi_1(a,b)
            =\epsilon (\rho(u,v)+\rho (z,w)-\rho(u,w)-\rho (z,v)).
        \end{align*}
        By the assumption that $\pi$ is optimal, the value of the above expression is non-positive. 
        As $\epsilon>0$, we obtain $\rho(u,v)+\rho (z,w)\le \rho(u,w)+\rho (z,v)$.
\end{proof}

We remark that the vertices $u,v,z,w$ in Lemma \ref{adjust} need not be distinct. Based on Lemma~\ref{simple}, we obtain the following lemmas on Lin--Lu--Yau curvature of an edge. 

\begin{lemma}\label{lowerbound}
    Let $G$ be a graph and $xy$ be an edge in $G$ with $d_x\ge d_y$.
    Then,
    $$\kappa_{LLY}(x,y)\ge \frac{2|N_{xy}|+2}{|N_{xy}|+|A_x|+1} -\frac{|N_{xy}|+2|A_y|}{|N_{xy}|+|A_y|+1}.$$
\end{lemma}

\begin{proof}
    For any fixed $p\in \left[ \frac{1}{1+d_y},1 \right)$, by Lemma \ref{simple}, there is a simple optimal transport plan $\pi$ from $\mu_x^p$ to $\mu_y^p$ such that 
    $$\pi(x,y)=p-\frac{1-p}{d_y}.$$ 
    For any two vertices $u\in A_x$ and $v\in A_y\cup N_{xy}$, we have $\rho(u,v)\le 3$ if $v\in A_y$, and $\rho(u,v)\le 2$ if $v\in N_{xy}$. 
    Thus, we derive
    \begin{align*}
        W(\mu_x^p,\mu_y^p)&= \sum_{u\in V}\sum_{v\in V}\rho(u,v)\pi(u,v)\\
        &= \sum_{u\in A_x}\sum_{v\in A_y\cup N_{xy}}\rho(u,v)\pi(u,v)+ \pi(x,y)+ \sum_{u\in A_x}\rho(u,y)\pi(u,y)\\
        &\le 3\sum_{u\in A_x}\sum_{v\in A_y}\pi(u,v)+ 2\sum_{u\in A_x}\sum_{v\in N_{xy}}\pi(u,v) + \pi(x,y)+ 2\sum_{u\in A_x}\pi(u,y)\\
        &=3|A_y|\frac{1-p}{d_y}+2|N_{xy}|\left(\frac{1-p}{d_y}-\frac{1-p}{d_x}\right)+ p-\frac{1-p}{d_y}+2\left(\frac{1-p}{d_y}-\frac{1-p}{d_x}\right).
    \end{align*}
    It follows by \eqref{eq:bourne} that
    \begin{align*}
        \kappa_{LLY}(x,y)=\frac{1-W(\mu_x^p,\mu_y^p)}{1-p}&\ge 1-\frac{2|N_{xy}|+3|A_y|+1}{d_y}+\frac{2|N_{xy}|+2}{d_x}\\
        &= \frac{2|N_{xy}|+2}{d_x}-\frac{|N_{xy}|+2|A_y|}{d_y}\\
        &= \frac{2|N_{xy}|+2}{|N_{xy}|+|A_x|+1}- \frac{|N_{xy}|+2|A_y|}{|N_{xy}|+|A_y|+1},
    \end{align*}
    completing the proof.
\end{proof}

\begin{lemma}\label{m2}
    Let $G=(V,E)$ be a graph. 
    Let $x$ and $y$ be two adjacent vertices in $G$ with $d_x\ge d_y$. 
    Let $p\in \left[ \frac{1}{1+d_y},1 \right)$, and let $\pi$ be a simple transport plan from $\mu_x^p$ to $\mu_y^p$ such that $\pi(x,y)=p-\frac{1-p}{d_y}.$ 
    Suppose that for any two vertices $u\in A_x$ and $v\in A_y$ with $\pi(u,v)>0$, we have $\rho(u,v)\le 2$. 
    Then,
    \begin{align*}
        \kappa_{LLY}(x,y)\ge \frac{1}{d_y}+\frac{|N_{xy}|+1}{d_x}-\frac{m_2}{1-p},
    \end{align*}
    where $$m_2:=\sum_{u\in V}\sum_{\rho(u,v)=2} \pi(u,v).$$
\end{lemma}
\begin{proof}
    Since $\pi$ is a simple transport plan, we have
\begin{align*}
    W(\mu_x^p,\mu_y^p)\le \sum\limits_{u\in V}\sum\limits_{v\in V}\pi(u,v)\rho(u,v)
    &=\sum\limits_{u\in V}\sum\limits_{\rho(u,v)=1}\pi(u,v)+2\sum\limits_{u\in V}\sum\limits_{\rho(u,v)=2}\pi(u,v)\\
    &=\sum\limits_{u\in V}\sum\limits_{v\ne u}\pi(u,v)+m_2.
\end{align*}
Note that
$$\sum\limits_{u\in V}\sum\limits_{v\ne u}\pi(u,v)=\pi(x,y)+|A_x|\frac{1-p}{d_x}=\left(p-\frac{1-p}{d_y} \right)+|A_x|\frac{1-p}{d_x}.$$
We obtain
    $$W(\mu_x^p,\mu_y^p)
    \le \left(p-\frac{1-p}{d_y} \right)+|A_x|\frac{1-p}{d_x}+m_2.$$
It follows by~\eqref{eq:bourne} that
$$\kappa_{LLY}(x,y)=\frac{1-W(\mu_x^p,\mu_y^p)}{1-p}\ge 1+\frac{1}{d_y}-\frac{|A_x|}{d_x}-\frac{m_2}{1-p}=\frac{1}{d_y}+\frac{|N_{xy}|+1}{d_x}-\frac{m_2}{1-p}.  $$
This completes the proof.
\end{proof}

\begin{corollary}\label{basic}
    Let $G$ be a graph. 
    Let $x$ and $y$ be two adjacent vertices in $G$ with $d_x\ge d_y$. 
    Let $p\in \left[ \frac{1}{1+d_y},1 \right)$, and let $\pi$ be a simple transport plan from $\mu_x^p$ to $\mu_y^p$ such that $\pi(x,y)=p-\frac{1-p}{d_y}.$ 
    Suppose that for any two vertices $u\in A_x$ and $v\in A_y\cup N_{xy}$ with $\pi(u,v)>0$, we have $\rho(u,v)=1$. Then, $\kappa_{LLY}(x,y)>0$.
\end{corollary}

\begin{proof}
    Note that
    $$\sum_{u\in V}\sum_{\rho(u,v)=2} \pi(u,v)=\sum_{u\in A_x}\pi(u,y)=\frac{1-p}{d_y}-\frac{1-p}{d_x}.$$
    Then, Lemma \ref{m2} gives $\kappa_{LLY}(x,y)\ge (|N_{xy}|+2)/d_x >0$.
\end{proof}

\section{\texorpdfstring{Complements of $C_4$-Free Graphs}{Complements of C4-Free Graphs}}

In this section, we prove Theorem \ref{main}.
\begin{proof}[Proof of Theorem \ref{main}]
    Let $G=(V,E)$ be a graph such that the complement graph $\overline{G}$ of $G$ contains no cycles of length 4.
    Assume that there is an edge $xy$ in $G$ such that the Lin--Lu--Yau curvature $\kappa_{LLY}(x,y)$ of $xy$ is non-positive. 
    Our aim is to show that $G$ is a path on 4 vertices.
    Without loss of generality, assume that $d_x\ge d_y$.
    Then, we have $|A_x|=d_x-|N_{xy}|-1\ge d_y-|N_{xy}|-1=|A_y|$.
    For any fixed $p\in \left[ \frac{1}{1+d_y},1 \right)$,  by Lemma \ref{simple}, there is a simple optimal transport plan from $\mu_x^p$ to $\mu_y^p$ such that $$\pi(x,y)=p-\frac{1-p}{d_y}.$$
    According to $\kappa_{LLY}(x,y)\le 0$ and Corollary \ref{basic}, there exist $z\in A_x$ and $w\in A_y\cup N_{xy}$ such that $\pi(z,w)>0$ and $\rho(z,w)\ge 2$.

    The following claim states that the vertex $z$ is unique.

    \begin{claim}\label{claim1}
        For any $u\in A_x\backslash\{z\}$ and $v\in A_y\cup N_{xy}$ such that $\pi(u,v)>0$, we have $\rho(u,v)=1$.
    \end{claim}
    \begin{proof}
        For a contradiction, assume that there exist $u_0\in A_x\backslash\{z\}$ and $v_0\in A_y\cup N_{xy}$ such that $\pi(u_0,v_0)>0$ and $\rho(u_0,v_0)\ge 2$. 
        If $w=v_0$, then $wzyu_0$ is a cycle of length $4$ in $\overline{G}$, which is a contradiction. 
        Thus, we have $w\ne v_0$. 
        If $z$ is not adjacent to $v_0$, then $v_0zyu_0$ is a cycle of length $4$ in $\overline{G}$. 
        Thus, $z$ is adjacent to $v_0$. 
        Similarly, considering $zwu_0y$ yields $wu_0\in E$. 
        Since $\pi(u_0,v_0)>0$, $\pi(z,w)>0$, and $\pi$ is an optimal transport plan, by Lemma~\ref{adjust}, we have 
        \begin{align*}
        4\leq \rho(u_0,v_0)+\rho (z,w)\le \rho(u_0,w)+\rho (z,v_0)=2, 
    \end{align*}
    which leads to a contradiction.
    \end{proof}

Note that the existence of $z$ implies $A_x\ne \emptyset$. 
We now divide the proof into two cases according to the size of $A_x$.

\noindent\textbf{Case 1:} We have $|A_x|= 1$.

If $|A_y|=0$, then Lemma \ref{lowerbound} gives $\kappa_{LLY}(x,y)\ge \frac{2|N_{xy}|+2}{|N_{xy}|+2}- \frac{|N_{xy}|}{|N_{xy}|+1}>0$. 
Thus, we have $|A_y|=1$, since $|A_x|\ge |A_y|$.  
Using Lemma \ref{lowerbound} again yields $\kappa_{LLY}(x,y)\ge \frac{|N_{xy}|}{|N_{xy}|+2}\ge 0$. 
It follows that $\kappa_{LLY}(x,y)=0$ and $|N_{xy}|=0$. 
That is, $A_x=\{z\}$ and $A_y=\{w\}$.
If $V= \{x,y,z,w\}$, then $G$ is a path on 4 vertices, and the proof is done.
Assume that $V\ne \{x,y,z,w\}$. 
Let $u_0$ be a vertex in $V\backslash \{x,y,z,w\}$. 
We next show that $u_0$ is adjacent to both $z$ and $w$. If $u_0$ is not adjacent to $z$, then $u_0zwx$ is a cycle of length 4 in $\overline{G}$, which is a contradiction. Hence, $u_0$ is adjacent to $z$. 
Similarly, $u_0$ is adjacent to $w$. So, $\rho(z,w)=2$.
Let $\pi_1$ be the transport plan defined as
        \[\pi_1(u,v)=\left\{
            \begin{array}{ll}
                \min\{\mu_x^p(u),\mu_y^p(u)\}, & \hbox{if $u=v$;}\\
                p-\frac{1-p}{2}, & \hbox{if $u=x$, $v=y$;} \\
                \frac{1-p}{2}, & \hbox{if $u=z$, $v=w$;} \\
                0, & \hbox{otherwise.}
            \end{array}
            \right.
        \]
By the definition of the Wasserstein distance, we derive
    \begin{align*}
        W(\mu_x^p,\mu_y^p)&\le \sum_{u\in V}\sum_{v\in V}\rho(u,v)\pi_1(u,v)
        = \left(p-\frac{1-p}{2}\right)+2\times\frac{1-p}{2}.
    \end{align*}
    It follows by~\eqref{eq:bourne} that
    \begin{align*}
        \kappa_{LLY}(x,y)=\frac{1-W(\mu_x^p,\mu_y^p)}{1-p}\ge \frac{1}{2},
    \end{align*}
    which contradicts the assumption that $\kappa_{LLY}(x,y)\le 0$.

\noindent\textbf{Case 2:} We have $|A_x|\ge 2$.

Recall that $z\in A_x$ and $w\in A_y\cup N_{xy}$ satisfy $\pi(z,w)>0$ and $\rho(z,w)\ge 2$. 

\begin{claim}\label{claim9}
    For any vertex $v\in A_x\backslash\{z\}$, we have $\pi(v,y)=0$. 
\end{claim}

\begin{proof}
    Assume that $v_0$ is a vertex in $A_x\backslash\{z\}$ such that $\pi(v_0,y)>0$. 
    If $w$ is not adjacent to $v_0$, then $wv_0yz$ forms a cycle of length 4 in $\overline{G}$, which is a contradiction.
    Hence, $w$ is adjacent to $v_0$.
    Since $\pi(v_0,y)>0$, $\pi(z,w)>0$, and $\pi$ is an optimal transport plan, by Lemma~\ref{adjust}, we have 
        \begin{align*}
        4\leq \rho(v_0,y)+\rho (z,w)\le \rho(v_0,w)+\rho (z,y)=3, 
    \end{align*}
    which is a contradiction.
\end{proof}

Now, combining Claim \ref{claim1} and Claim \ref{claim9}, we derive
\begin{align*}
        W(\mu_x^p,\mu_y^p)&= \sum_{u\in V}\sum_{v\in V}\rho(u,v)\pi(u,v)\\
        &= \sum_{u\in A_x\backslash\{z\}}\sum_{v\in A_y\cup N_{xy}}\rho(u,v)\pi(u,v)+ \sum_{v\in A_y\cup N_{xy}\cup \{y\}}\rho(z,v)\pi(z,v)+ \pi(x,y)\\
        &\le \sum_{u\in A_x\backslash\{z\}}\sum_{v\in A_y\cup N_{xy}}\pi(u,v)+ 3\sum_{v\in A_y\cup N_{xy}\cup \{y\}}\pi(z,v)+ \pi(x,y)\\
        &=(|A_x|-1)\frac{1-p}{d_x}+3\times\frac{1-p}{d_x}+ \left(p-\frac{1-p}{d_y}\right)
\end{align*}
It follows by~\eqref{eq:bourne} that
    \begin{align*}
        \kappa_{LLY}(x,y)=\frac{1-W(\mu_x^p,\mu_y^p)}{1-p}&\ge 1-\frac{|A_x|+2}{d_x}+\frac{1}{d_y}= \frac{|N_{xy}|-1}{d_x}+\frac{1}{d_y}\ge \frac{1}{d_y}-\frac{1}{d_x}\ge 0.
    \end{align*}
Since $\kappa_{LLY}(x,y)\le 0$, we have $\kappa_{LLY}(x,y)=0$, and the equality implies $d_x=d_y$.
So, $|A_y|=|A_x|\ge 2$.
Take a vertex $v_0\in A_y\backslash \{w\}$. 
Since $\pi(z,w)>0$, we find $\pi(z,v_0)<\mu^p_x(z)\le \mu^p_y(v_0)$. Thus, there exists $u_0\in A_x$ such that $\pi(u_0,v_0)>0$.
If $u_0w\notin E$, then $u_0wzy$ forms a cycle of length 4 in $\overline{G}$, which is a contradiction. So, $u_0w\in E$. Similarly, considering $zv_0xw$ results in $zv_0\in E$.
Applying Lemma \ref{adjust} with vertices $z,w,u_0,v_0$ yields
$$3\leq \rho(z,w)+\rho (u_0,v_0)\le \rho(z,v_0)+\rho (u_0,w)=2,$$
which is impossible. This completes the proof.
\end{proof}

\section{\texorpdfstring{Complements of $K_{2,t}$-Free Graphs}{Complements of K2t-Free Graphs}}
In this section, we prove Theorem~\ref{main2}.

\begin{proof}[Proof of Theorem \ref{main2}]
    Let $G=(V,E)$ be a graph on $n$ vertices such that $n\ge \max\{t^2-2t +2,8t\}$ and the complement graph $\overline{G}$ of $G$ contains no $K_{2,t}$.
    For a contradiction, assume that there is an edge $xy$ in $G$ such that the Lin--Lu--Yau curvature $\kappa_{LLY}(x,y)$ of $xy$ is non-positive. 
    Without loss of generality, assume that $d_x\ge d_y$. 
    Then, $|A_x|\ge |A_y|$. 

    \begin{claim}\label{sum}
        We have $n\le d_x+d_y-|N_{xy}|+t-1$.
    \end{claim}
    \begin{proof}
        Since $\overline{G}$ contains no $K_{2,t}$, we find $|R_{xy}|\le t-1$. 
        So, $n\le d_x+d_y-|N_{xy}|+t-1$.
    \end{proof}

    \begin{claim}\label{Claim}
        For any vertex $v\in V\backslash \{x,y\}$, $v$ has at most $t-1$ non-neighbors in $A_x\cup R_{xy}$ and at most $t-1$ non-neighbors in $A_y\cup R_{xy}$.
    \end{claim}

    \begin{proof}
        If there exists a vertex $v\in V\backslash \{x,y\}$ with at least $t$ non-neighbors in $A_x\cup R_{xy}$, then these $t$ non-neighbors together with $v$ and $y$ form a $K_{2,t}$ in $\overline{G}$, which is a contradiction. 
        Similarly, any vertex $v\in V\backslash \{x,y\}$ has at most $t-1$ non-neighbors in $A_y\cup R_{xy}$.
    \end{proof}
    
    For any fixed $p\in \left[ \frac{1}{1+d_y},1 \right)$, it follows by Lemma \ref{simple} that there is a simple optimal transport plan from $\mu_x^p$ to $\mu_y^p$ such that 
    \begin{align}\label{assumption}
        \pi(x,y)=p-\frac{1-p}{d_y}.
    \end{align}

        We now prove that there is an injection $\phi_1$ from $A_y$ to $A_x$ such that $\pi(\phi_1(v),v)>0$ for any vertex $v\in A_y$.
        For any subset $S\subseteq A_y$, define $N(S):=\{ u\in A_x\ |\ \exists v\in S,\ \pi(u,v)>0 \}$.
        According to Theorem \ref{Marriage}, it suffices to show that $|N(S)|\ge |S|$ for any $S\subseteq A_y$. 
        By the choice of $\pi$, we have 
        $$\frac{(1-p)|N(S)|}{d_x}=\sum_{u\in N(S)}\mu_x^p(u)\ge \sum_{v\in S}\mu_y^p(v)=\frac{(1-p)|S|}{d_y}.$$
        By the assumption that $d_x\ge d_y$, we have $|N(S)|\ge |S|$. 
        Thus, such an injection $\phi_1$ exists.

    \begin{claim}\label{Ay}
        We have $|A_y|\le 2t.$
    \end{claim}

    \begin{proof}
        By Corollary \ref{basic} and the assumption that $\kappa_{LLY}(x,y)\le 0$, there exist $z\in A_x$ and $w\in A_y\cup N_{xy}$ such that $\pi(z,w)>0$ and $\rho(z,w)\ge 2$.
        
        We first show that, for any $v\in A_y\backslash \{w, \phi_1^{-1}(z)\}$, we have $\rho(z,v)\ge 2$ or $\rho(w,\phi_1(v))\ge 2$.
        For a contradiction, assume that there exists $v\in A_y\backslash \{w, \phi_1^{-1}(z)\}$ such that $zv\in E$ and $w\phi_1(v)\in E$.
        Then, applying Lemma~\ref{adjust} with vertices $\phi_1(v),v,z,w$ leads to
        \begin{align*}
        3\leq \rho(\phi_1(v),v)+\rho (z,w)\le \rho(\phi_1(v),w)+\rho (z,v)=2, 
    \end{align*}
    which is a contradiction.

    Therefore, the sum of the number of non-neighbors of $z$ in $A_y$ and the number of non-neighbors of $w$ in $A_x$ is at least $|A_y\backslash \{w, \phi_1^{-1}(z)\}|\ge |A_y|-2$.
    On the other hand, Claim \ref{Claim} implies that this sum is at most $2(t-1)$. 
    It follows that $|A_y|\le 2t$.
    \end{proof}
    
    \begin{claim}\label{Claim3}
        For any $u\in A_x$ and $v\in N_{xy}\cup A_y$, we have $\rho(u,v)\le 2$.
    \end{claim}

    \begin{proof}
    We first prove that $|A_x\cup R_{xy}|\ge 2t$. By Lemma~\ref{lowerbound}, we have
    \begin{align*}
        0\ge \kappa_{LLY}(x,y)&\ge \frac{2|N_{xy}|+2}{|N_{xy}|+|A_x|+1} -\frac{|N_{xy}|+2|A_y|}{|N_{xy}|+|A_y|+1}\\
        &= \frac{2|N_{xy}|+2}{|N_{xy}|+|A_x|+1} -2+\frac{|N_{xy}|+2}{|N_{xy}|+|A_y|+1}\ge \frac{3|N_{xy}|+4}{|N_{xy}|+|A_x|+1} -2,
    \end{align*}
    which implies $|N_{xy}|\le 2|A_x|-2\le 2|A_x\cup R_{xy}|-2$. 
    Since $|A_y|\le |A_x|$, we derive
    $$n=2+|A_x\cup R_{xy}|+|N_{xy}|+|A_y|\le 4|A_x\cup R_{xy}|.$$
    By the assumption that $n\ge 8t$, we find $|A_x\cup R_{xy}|\ge 2t$. 
    
        For a contradiction, assume that there are two vertices $u_0\in A_x$ and $v_0\in N_{xy}\cup A_y$ with $\rho(u_0,v_0)\ge 3$.
        Then, each vertex in $(A_x\cup R_{xy})\backslash \{u_0\}$ is not adjacent to at least one of $u_0$ and $v_0$.
        Since $|A_x\cup R_{xy}|\ge 2t$, at least one of $u_0$ and $v_0$ is not adjacent to at least $t$ vertices in $(A_x\cup R_{xy})\backslash \{u_0\}$, which contradicts Claim \ref{Claim}.
    \end{proof}

    Set $X_1:=\{u\in A_x\ |\ \exists\, v\in V,\ \pi(u,v)>0,\ \rho(u,v)= 2 \}$, $Y_1:=N_{N_{xy}\cup A_y}(X_1)$, $X_2:=\{u\in A_x\ |\ \exists\, v\in Y_1, \pi(u,v)>0 \}$, and $Y_2:=\{v\in N_{xy}\cup A_y\ |\ \exists\, u\in A_x,\ \pi(u,v)>0,\ \rho(u,v)= 2 \}$.

    \begin{claim}\label{overlap}
        We have $Y_1\cap Y_2=\emptyset$.
    \end{claim}
    \begin{proof}
        Assume, for a contradiction, that there exist $u\in A_x$ and $v\in Y_1$ with $\pi(u,v)>0$ and $\rho(u,v)= 2$.
        By the definition of $Y_1$, there exists $z\in X_1$ such that $zv\in E$.
        By the definition of $X_1$, there exists $w\in V$ such that $\pi(z,w)>0$ and $\rho(z,w)=2$. 
        Applying Lemma \ref{adjust} with $z,w,u,v$ yields 
        $$4= \rho(u,v)+\rho (z,w)\le \rho(u,w)+\rho (z,v)\le 3,$$
        which is a contradiction.
    \end{proof}

    \begin{claim}\label{count}
        For any $w\in Y_2$ and $u\in X_1\cup X_2$, we have $uw\notin E$.
    \end{claim}

    \begin{proof}
        For a contradiction, assume that $uw\in E$.
        By the definition of $Y_2$, there exists $z\in X_1$ such that $\pi(z,w)>0$ and $\rho(z,w)=2$.
        If $u\in X_1$, then $w\in Y_1$, which is contradictory to Claim \ref{overlap}. 
        So, $u\in X_2$.
        Then there exists $v\in Y_1$ and $z'\in X_1$ such that $\pi(u,v)>0$ and $z'v\in E$. 
        Since $z'\in X_1$, there exists $w'\in V$ such that $\pi(z',w')>0$ and $\rho(z',w')=2$.
        Let $\epsilon:=\frac{1}{3}\min\{\pi(u,v),\pi(z,w),\pi(z',w')\}$. 
        Then $\epsilon>0$. 
        Consider the transport plan $\pi_1$ defined as $$\pi_1:=\pi-\epsilon(\delta_{zw}+\delta_{uv}+\delta_{z'w'}-\delta_{zw'}-\delta_{uw}-\delta_{z'v}).$$
        It follows that
        \begin{align*}
            &\sum_{a,b\in V}\rho(a,b)\pi(a,b)-\sum_{a,b\in V}\rho(a,b)\pi_1(a,b)\\
            =\ &\epsilon (\rho(z,w)+\rho (u,v)+\rho(z',w')-\rho (z,w')-\rho (u,w)-\rho (z',v))\\
            \ge\ &\epsilon (2+1+2-2-1-1)>0,
        \end{align*}
        which is contradictory to the assumption that $\pi$ is optimal.
    \end{proof}
    
    According to Claim \ref{Claim3}, Lemma \ref{m2}, and the assumption that $\kappa_{LLY}(x,y)\le 0$, we deduce that
    \begin{align}\label{eqm2}
        \frac{1-p}{d_y}+\frac{(|N_{xy}|+1)(1-p)}{d_x}\le m_2,
    \end{align}
    where $m_2$ is defined as in Lemma \ref{m2}.
    
    Note that Corollary \ref{basic} implies $Y_2\ne \emptyset$. Let $w_1$ be a vertex in $Y_2$. By Claim \ref{count}, $w_1$ has at least $|X_1\cup X_2|$ non-neighbors in $A_x$.
    Then Claim \ref{Claim} gives
    \begin{align}\label{t1}
        |X_1\cup X_2|\le t-1.
    \end{align}

    According to Claim \ref{overlap}, all mass transported into $Y_1$ is transported over distance 1.
    Therefore, by the definitions of $X_1$ and $X_2$, we derive
    \begin{align*}
        |X_1\cup X_2|\frac{1-p}{d_x}=\sum_{u\in X_1\cup X_2}\sum_{v\in V} \pi(u,v)\ge m_2+\sum_{u\in X_2}\sum_{v\in Y_1} \pi(u,v).
    \end{align*}
    Combined with inequalities \eqref{eqm2} and \eqref{t1}, we have
    \begin{align}\label{mid}
        \frac{(t-1)(1-p)}{d_x}\ge \frac{1-p}{d_y}+\frac{(|N_{xy}|+1)(1-p)}{d_x}+ \sum_{u\in X_2}\sum_{v\in Y_1} \pi(u,v).
    \end{align}
    Note that
    $$\sum_{u\in X_2}\sum_{v\in Y_1} \pi(u,v)=|Y_1\backslash N_{xy}|\frac{1-p}{d_y}+|Y_1\cap N_{xy}| \left(\frac{1-p}{d_y}-\frac{1-p}{d_x}\right) \ge |Y_1|\frac{1-p}{d_y}-|N_{xy}|\frac{1-p}{d_x}.$$
    It follows that 
    \begin{align*}
        \frac{(t-1)(1-p)}{d_x}\ge \frac{1-p}{d_y}+\frac{(|N_{xy}|+1)(1-p)}{d_x}+ |Y_1|\frac{1-p}{d_y}-|N_{xy}|\frac{1-p}{d_x},
    \end{align*}
    which is equivalent to
    \begin{align}\label{eq1}
        d_x\le \frac{(t-2)d_y}{|Y_1|+1}.
    \end{align}

    \begin{claim}\label{X1}
        We have $|X_1|\ge \dfrac{d_x}{d_y}+|N_{xy}|+1.$
    \end{claim}
    \begin{proof}
        By the definition of $X_1$, we have 
    $$m_2\le \sum_{u\in X_1} \sum_{v\in V}\pi(u,v)= \frac{|X_1|(1-p)}{d_x},$$
    which together with inequality \eqref{eqm2} gives 
    the desired estimate.
    \end{proof}
    
    Let $z$ and $z'$ be any two vertices in $X_1$ (By Claim \ref{X1}, such two vertices exist). By the definition of $Y_1$, every vertex in $(N_{xy}\cup A_y)\backslash Y_1$ is a common non-neighbor of $z$ and $z'$.
    Since $\overline{G}$ contains no $K_{2,t}$, $z$ and $z'$ have at most $t-2$ common non-neighbors except $y$. Hence, we have $|N_{xy}\cup A_y|-|Y_1|\le t-2$. That is,
    \begin{align}\label{n12}
        d_y\le |Y_1|+t-1.
    \end{align}
    
    Recall from Claim \ref{sum} that $n\le d_x+d_y-|N_{xy}|+t-1$.
    Combined with inequalities \eqref{eq1} and \eqref{n12}, this leads to
    \begin{align}\notag
        n&\le \left(\frac{t-2}{|Y_1|+1}+1 \right)(|Y_1|+t-1)-|N_{xy}|+t-1\\ \label{eq3}
        &= \frac{(t-2)^2}{|Y_1|+1}+|Y_1|-|N_{xy}|+3t-4.
    \end{align}
    
    \begin{claim}\label{Y1}
        We have $Y_1=\emptyset$.
    \end{claim}
    \begin{proof}
        For a contradiction, assume that $|Y_1|\ge 1$.
        Note that the value of expression \eqref{eq3} decreases as $|Y_1|$ increases when $|Y_1|\le t-3$.
         So, if $|Y_1|\le t-3$, then inequality \eqref{eq3} implies 
            $$n\le \frac{(t-2)^2}{1+1}+1-|N_{xy}|+3t-4< \max\{t^2-2t+2, 3t\},$$  
        contradicting the lower bound of $n$. Thus, $|Y_1|> t-3$. Recall from Claim \ref{Ay} that $|A_y|\le 2t$. Since $Y_1\subseteq N_{xy}\cup A_y$, we find $|Y_1|-|N_{xy}|\le 2t$. Thus, by inequality \eqref{eq3}, we derive
        $$n< (t-2)+2t+3 t-4\le \max\{t^2-2t+2, 5t\},$$
        also a contradiction. This completes the proof.
    \end{proof}

        We now prove that $|A_y|\ge 2$.
        By inequality \eqref{mid}, we have
        \begin{align*}
        \frac{(t-1)(1-p)}{d_x}\ge \frac{1-p}{d_y}+\frac{(|N_{xy}|+1)(1-p)}{d_x},
    \end{align*}
    which is equivalent to $d_x\le (t-|N_{xy}|-2)d_y$.
    Substituting it into Claim \ref{sum} yields
    \begin{align*}
        n\le (t-|N_{xy}|-1)d_y-|N_{xy}|+t-1.
    \end{align*}
    If $|A_y|\le 1$, then $d_y\le |N_{xy}|+2$, and hence
    \begin{align*}
        n&\le (t-|N_{xy}|-1)(|N_{xy}|+2)-|N_{xy}|+t-1\\ 
        &=-\left(|N_{xy}| - \frac{t}{2} + 2\right)^2 + \frac{(t + 2)}{4}^2\le\frac{(t + 2)}{4}^2< \max\{t^2-2t+2, 3t\},
    \end{align*}
    which contradicts the lower bound on $n$. 
    Thus, $|A_y|\ge 2$.

    Let $w_2$ and $w_3$ be two vertices in $A_y$ (Since $|A_y|\ge 2$, such two vertices exist). By Claim \ref{Y1} and the definition of $Y_1$, every vertex in $X_1$ is a common non-neighbor of $w_2$ and $w_3$. Since $\overline{G}$ has no $K_{2,t}$, $w_2$ and $w_3$ have at most $t-2$ common non-neighbors except $x$. Thus, we obtain $|X_1|\le t-2$. 
    It follows by Claim \ref{X1} that
    \begin{align*}
        \frac{d_x}{d_y}+|N_{xy}|+1 \le t-2.
    \end{align*}
    That is, $d_x\le (t-|N_{xy}|-3)d_y$. Substituting it into Claim \ref{sum} yields
    \begin{align*}
        n\le (t-|N_{xy}|-2)d_y-|N_{xy}|+t-1.
    \end{align*}
    By Claim \ref{Y1} and inequality \eqref{n12}, we have $d_y\le t-1$. It follows that
    \begin{align*}
        n\le (t-|N_{xy}|-2)(t-1)-|N_{xy}|+t-1\le t^2-2t+1,
    \end{align*}
    which contradicts the lower bound of $n$. This completes the proof.
\end{proof}

\section{Sharpness discussions}\label{section:sharpness}

In this section, we present some examples to illustrate that our results cannot be extended to the more general cases.

Let $G=(V,E)$ be the graph  of order $n$ defined as follows.
The vertex set \[V:=\{x,y\}\sqcup A\sqcup B\sqcup C\sqcup D,\] where $|A|=a$, $|B|=b$, $|C|=c$, $|D|=d$ such that $a+b+c+d+2=n$.
The edge set \[E:=\{xy\}\sqcup E_1\sqcup E_2\sqcup E_3 \sqcup E_4,\] where 
\begin{align*}
    &E_1:=\{xv | v\in A \cup B\}\cup\{uv | u, v\in A\cup B\text{ and } u\neq v\},\\
    &E_2:=\{yv | v\in C\}\cup\{uv | u, v\in C\text{ and } u\neq v\},\\
    &E_3:=\{uv| u\in B,\ v\in C\},\\
    &E_4:=\{uv | u\in D,\ v\in A\cup B\cup C\}\cup \{uv| u, v\in D \text{ and } u\neq v\}.
\end{align*}
In particular, the induced subgraphs of the vertex sets $A, B, C,D$ are isomorphic to complete graphs $K_a, K_b, K_c, K_d$, respectively. A schematic of $G$ is depicted in Figure \ref{fig:enter-label}.
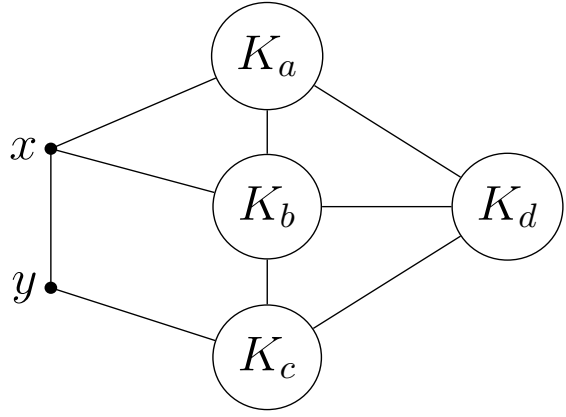
\begin{figure}[htbp]
    \centering
    \begin{tikzpicture}[scale=1.2, every node/.style={font=\Large}]
\coordinate (x) at (0.4,2);    
\coordinate (y) at (0.4,0.8);
\node[draw,circle,minimum size=1cm] (Ka) at (2.2,2.8) {$K_a$};    
\node[draw,circle,minimum size=1cm] (Kb) at (2.2,1.5) {$K_b$};    
\node[draw,circle,minimum size=1cm] (Kc) at (2.2,0.2) {$K_c$};    
\node[draw,circle,minimum size=1cm] (Kd) at (4.2,1.5) {$K_d$};
\fill (x) circle (1.5pt);    
\fill (y) circle (1.5pt);
\node[left] at (x) {$x$};    
\node[left] at (y) {$y$};
\draw (x) -- (y);
\draw (x) -- (Ka);    
\draw (x) -- (Kb);
\draw (y) -- (Kc);
\draw (Ka) -- (Kb);    
\draw (Kb) -- (Kc);
\draw (Ka) -- (Kd);    
\draw (Kb) -- (Kd);    
\draw (Kc) -- (Kd);
\end{tikzpicture}
    \caption{A schematic of the graph $G$.}
    \label{fig:enter-label}
\end{figure}

We next estimate the Lin--Lu--Yau curvature $\kappa_{LLY}(x,y)$ of $xy$.
Consider the function $f:N_V(x)\cup N_V(y)\rightarrow\mathbb{Z}$ given by
$$f(z)= 
\begin{cases}
-1, & \text { if } z\in A;\\
0, & \text { if } z\in \{x\} \cup B \cup D; \\
1, & \text { if } z\in \{y\} \cup C.
\end{cases}$$
Then, $f(y)-f(x)=1$ and $f\in Lip(1)$.
By Theorem~\ref{Curvature via the Laplacian}, we have 
\[\kappa_{LLY}(x,y)\leq \frac{1-a}{a+b+1}+\frac{1}{c+1}.\]

Choose $a=\lceil\frac{n}{2}\rceil$, $b= \lfloor\frac{n}{2}\rfloor-3$, $c=1$, $d=0$. 
Then, we have $\kappa_{LLY}(x,y)\leq 0$. 
Then every cycle in $\overline{G}$ has length $4$. 
This shows that Theorem~\ref{main} cannot be generalized by forbidding cycles of length $k$ in the complement graph for any $k=3$ or $k>4$. 
Moreover, the complement of $G$ contains no $K_{s,t}$ with $s>2$ and $t> 2$. 
Therefore, Theorem~\ref{main2} cannot be generalized by forbidding $K_{s,t}$ for $s>2$ and $t> 2$.

For fixed $t$, choose $a=c=t-2$, $b=n-3t+3$, $d=t-1$. 
Then we have $\kappa_{LLY}(x,y)\leq 0$ when $n\leq t^2-2t+1$.
Moreover, the complement of $G$ contains no $K_{2,t}$.
This shows that the bound $n\ge \max\{t^2-2t +2,8t\}$ in Theorem~\ref{main2} is sharp for $t\geq 10$.


\section*{Acknowledgement}
\noindent This work is supported by the National Key R \& D Program of China 2023YFA1010200. K.C.'s research is supported by the National Natural Science Foundation of China No. 125B1009 and the New Lotus Scholars Program PB22000259. S.L.'s research is supported by the Scientific Research Innovation Capability Support Project for Young Faculty and the National Natural Science Foundation of China No. 12431004.

\end{document}